\documentclass[11pt]{article}
\usepackage[final]{epsfig}
\usepackage[left=3cm,right=3cm, top=2.5cm,bottom=2.5cm,bindingoffset=0cm]{geometry}
\usepackage{graphics}
\usepackage{amsmath}
\usepackage{amsfonts}
\usepackage{latexsym}
\usepackage{amssymb, mathabx}
\usepackage{amsthm}
\usepackage{graphicx}
\usepackage{epstopdf}
\usepackage{multicol,multirow}
\usepackage{hyperref, enumitem}
\usepackage{mathtools}

\usepackage{tabularx}
\usepackage{array}

\usepackage{subcaption}

\usepackage{cancel}

\usepackage{url}
\usepackage{epstopdf}
\usepackage{wasysym}
\usepackage[OT2, T1]{fontenc}
\usepackage[russian, english]{babel}

\usepackage{epigraph}

 \usepackage{relsize}

\usepackage[nottoc]{tocbibind}

\usepackage{amsthm}

\usepackage{tikz}
\usetikzlibrary{
  arrows.meta,           
  positioning,           
  decorations.pathmorphing, 
  decorations.markings,  
  decorations.text       
}

\tikzset{>=Stealth}

\usetikzlibrary{shapes.geometric}

\tikzset{
  pil/.style={
    ->,
    thick,
    shorten <=2pt,
    shorten >=2pt,
  }
}

\tikzset{
  ->-/.style={
    decoration={markings, mark=at position 0.7 with {\arrow{>}}},
    postaction={decorate}
  }
}

\tikzset{
  -<-/.style={
    decoration={markings, mark=at position 0.4 with {\arrow{<}}},
    postaction={decorate}
  }
}

\tikzset{
  anglearrow/.style={
    decoration={markings, mark=at position 0.52 with {\arrow{angle 90}}},
    postaction={decorate}
  }
}

\usepackage{pgfplots}
\pgfplotsset{compat=1.18}

\makeatletter
\def\thmhead@plain#1#2#3{%
  \thmname{#1}\thmnumber{\@ifnotempty{#1}{ }\@upn{#2}}%
  \thmnote{ {\the\thm@notefont#3}}}
\let\thmhead\thmhead@plain

\makeatother

\newcounter{AppCounter}

\def\restrict#1{\raise-.5ex\hbox{\ensuremath|}_{#1}}

\newtheorem{lemma}{Lemma}[section]
\newtheorem{proposition}[lemma]{Proposition}

\newtheorem{remark-definition}[lemma]{Remark-Definition}
\newtheorem{theorem}[lemma]{Theorem}

\newtheorem{proposition-conjecture}[lemma]{Proposition-conjecture}

\theoremstyle{definition}

\newtheorem{definition}[lemma]{Definition}
\newtheorem{remark}[lemma]{Remark}

\newcommand{\grad}[1]{\nabla #1}

\newcommand{\id}{\mathrm{id}}

\newcommand{\low}[1]{\raise-.0ex\hbox{$\scriptstyle #1$}}
\newcommand{\high}[1]{\raise.5ex\hbox{$\scriptstyle #1$}}

\newcommand{\marginnote}[1]
{
}

\newcounter{ai}

\newcounter{bk}

\title {Infinite-dimensional  nonholonomic and vakonomic systems}
\date{~}

\author{Alexander G. Abanov\thanks{
    Department of Physics and Astronomy, Stony Brook University;
    e-mail:  \texttt{alexandre.abanov@stonybrook.edu}}~
and Boris Khesin\thanks{
    Department of Mathematics,
    University of Toronto, Canada;
    e-mail: \texttt{khesin@math.toronto.edu}}
}

\begin{document}
\maketitle

\vspace{-1cm}

\epigraph{You know that in moments of stress\\
You tend to get tenser, not less;\\
But since stress is a tensor\\
You needn't feel denser;\\
It's tricky, I have to confess.}{H.K.~Moffatt, FRS}

\begin{abstract}
In this paper, we present a collection of infinite-dimensional systems with nonholonomic constraints. 
In finite dimensions the two essentially different types of dynamics, nonholonomic or vakonomic ones, are known to be obtained by taking certain limits of holonomic systems with Rayleigh dissipation, as in \cite{kozlov1983realization}.
We visualize this phenomenon  for the classical example of a skate on an inclined plane.

The infinite-dimensional examples of nonholonomic and vakonomic systems revisited in the paper
include subriemannian and Euler–Poincaré–Suslov systems on  Lie groups, the Heisenberg chain, the general Camassa–Holm equation, infinite-dimensional geometry of a nonholonomic Moser theorem, subriemannian approximations of an ideal
hydrodynamics, parity-breaking nonholonomic fluids, and potential solutions to Burgers-type equations arising in optimal mass transport.
Finally, we return to a higher-dimensional analogue of the skate, the kinematics of a car with $n$ trailers, as well as its limit as
$n\to \infty$. We show that its infinite-dimensional version  is a snake-like motion of the 
Chaplygin sleigh with a string, and it is subordinated to  an infinite-dimensional Goursat distribution.

\end{abstract}

\tableofcontents



\section{Introduction}

The theory of infinite-dimensional Hamiltonian systems is by now  well developed, with plenty of examples and several possible frameworks for infinite-dimensional Poisson brackets. The corresponding evolution PDEs include the Korteweg--de Vries, Nonlinear Schr\"odinger, Camassa--Holm, Kadomtsev--Petviashvili and many other equations, while examples of infinite-dimensional symplectic and Poisson structures include those of Marsden--Weinstein ones on the space of knots and membranes, Gelfand--Dickey brackets on pseudo-differential symbols, Lie-Poisson brackets on the duals of infinite-dimensional Lie algebras, etc. 
On the other hand, infinite-dimensional contact (or more generally, nonholonomic) systems are very rare, despite the fact that in finite dimensions contact geometry is regarded as the twin sister of symplectic geometry. 
General equations of nonholonomic dynamics in an infinite-dimensional setting were described in  \cite{Shi2017, Shi2020}. 
Here, we collect several examples of infinite-dimensional nonholonomic and vakonomic systems that, in our opinion, are suggestive for the future development of infinite-dimensional nonholonomic mechanics.

The goal of this note is mostly expository. We start by describing two different types of nonholonomic dynamics, vakonomic 
or the nonholonomic one governed by the Lagrange-d'Alembert principle. They are known to be obtained by taking certain limits of holonomic systems with Rayleigh dissipation, see \cite{kozlov1983realization}. 
Vakonomic mechanics is closely related to subriemannian geometry and control theory, see \cite{montgomery2002tour}. In a nutshell, given a bracket-generating distribution on a manifold and a Lagrangian $L$ of a physical system  (more generally, any cost function), vakonomic approach describes trajectories as critical points (minimizers)
of the functional $\mathcal L=\int L \,dt$  on the set of admissible paths, i.e., paths subordinated to the given distribution  \cite{kozlov1992problem}. Such trajectories, having the variational origin, may drastically differ from the dynamics given by the Lagrange--d'Alembert principle, defining the motion of nonholonomic systems and requiring that the constraint forces would do no work on virtual displacements consistent with the
constraints \cite{bloch2003nonholonomic}.   Those two principles coincide for holonomic systems. 
In the papers \cite{kozlov1983realization, kozlov1992problem}, a general setting was described in which, by introducing a regularization via Rayleigh dissipation and taking different limits, one is led to either nonholonomic or vakonomic equations. We start by
reviewing and illuminating with figures  the classical example of a skate on an inclined plane considered in \cite{kozlov1983realization}.

This is followed by infinite-dimensional examples that include subriemannian and Euler–Poincaré–Suslov systems on Lie groups, and in particular, the Heisenberg chain and the general Camassa–Holm equation. Next we describe in more details the infinite-dimensional geometry of a nonholonomic Moser theorem and parity-breaking nonholonomic fluids, as well as   nonholonomic approximations of the Eulerian ideal hydrodynamics and potential solutions to Burgers-type equations arising in optimal mass transport. 

After that we return to the skate example in the context of Goursat distributions and the $n$-trailer systems. 
The reader might find it interesting to compare the many-trailer systems for a unicycle or skate 
with those of a car, which turns out to be related to a dimensional shift for the corresponding configuration space. 
Finally, we present the kinematics of its limit as
$n\to \infty$. We show that its infinite-dimensional version  is a snake-like motion of the 
Chaplygin sleigh with a string and it is subordinated to  an infinite-dimensional Goursat distribution.

For the reader's convenience, Table~\ref{tab:intro-summary} summarizes the main classes of constrained dynamics considered in the paper, together with the corresponding  geometric settings and the types of symmetry.

\begin{table}[t]
\centering
\caption{Examples of constrained dynamics discussed below}
\label{tab:intro-summary}
\renewcommand{\arraystretch}{1.12}
\setlength{\tabcolsep}{5pt}
\begin{tabularx}{\textwidth}{
>{\raggedright\arraybackslash}p{4.6cm}
>{\raggedright\arraybackslash}p{4.2cm}
>{\raggedright\arraybackslash}X}
\hline
\textbf{Type of dynamics} & \textbf{Type of symmetry} & \textbf{Systems} \\
\hline
Lagrange-d'Alembert principle
&Lie group $G$ 
&Euler-Poincar\'e-Suslov \hyperref[sect:Euler--Arnold]{\S\ref*{sect:Euler--Arnold}}\\
& stochastic matrices & Lagrange-d'Alembert approximations of hydrodynamics \hyperref[sec:fluidapprox]{\S\ref*{sec:fluidapprox}}\\
& Goursat distributions & $n$-trailer systems \hyperref[sec:ntrailers]{\S\ref*{sec:ntrailers}}\\
& & Car parking \hyperref[sec:parkingcar]{\S\ref*{sec:parkingcar}}\\
& infinite-dimensional Goursat distribution& Snake motions \hyperref[sec:snake]{\S\ref*{sec:snake}}\\
\hline
Subriemannian/vakonomic systems
&Lie group $G$ 
&Euler-Arnold equations \hyperref[sect:Euler--Arnold]{\S\ref*{sect:Euler--Arnold}}\\
& $\widetilde{ {\rm Diff}(S^1)}$ & Camassa-Holm equation \hyperref[sec:camassa]{\S\ref*{sec:camassa}}\\
& ${\rm Diff}(M)$ & Nonholonomic Moser theorem \hyperref[sec:Moser]{\S\ref*{sec:Moser}}\\
& & Visual cortex \hyperref[sec:visualcortex]{\S\ref*{sec:visualcortex}}\\
& ${\rm Diff}_\mu(M)$ & Subriemannian approximations of hydrodynamics \hyperref[sec:fluidapprox]{\S\ref*{sec:fluidapprox}}\\
\hline
Both/quasiholonomic systems
&$C^\infty(S^1,{\rm SO}(3))$ 
& Heisenberg chain \hyperref[sec:heisenberg]{\S\ref*{sec:heisenberg}}\\
& ${\rm Diff}(M)$ & Potential Burgers equation \hyperref[sec:burgers]{\S\ref*{sec:burgers}}\\
\hline
Kozlov's interpolation
& Rayleigh dissipation& Ideal skate \hyperref[sec:skate]{\S\ref*{sec:skate}}\\
& ${\rm Diff}(M)\ltimes V$ 
& Parity breaking fluid  \hyperref[sec:parityfluid]{\S\ref*{sec:parityfluid}}\\
\hline
\end{tabularx}
\end{table}


\section{Vakonomic and nonholonomic systems: the ideal skate problem}\label{sect:skate}

Let us start with a finite-dimensional holonomic natural system, depending on parameters. By appropriately introducing 
the Rayleigh dissipation and taking various limits one can obtain  vakonomic or nonholonomic systems, see \cite{kozlov1983realization, kozlov1992problem}.

\subsection{Various limits of natural systems}

Let $M$ be a Riemannian  manifold with a metric $g$ and $\tau$ a bracket-generating distribution on $M$. 
We assume the bracket-generating property of $\tau$ so that there existed admissible paths (i.e., tangent to the distribution $\tau$) connecting any two points in $M$.

Consider a natural system $\mathcal L(\gamma):=\int_{\gamma} L(q, \dot q) \, dt$ for the Lagrangian $L=K-U$, where
$K=(1/2)g(\dot q, \dot q)$ is the kinetic energy corresponding to the metric $g$ and $U=U(q(t))$, while the integral is taken over the paths $\gamma=\{\gamma(t), \, t\in [0,1]\}$ with fixed endpoints.

{\it Vakonomic system} describes extremals of the functional $\mathcal L$ among admissible paths, i.e., among paths $\gamma$ subordinated to the distribution $\tau$. It has a variational origin, hence the name.

Solutions of a {\it nonholonomic system} satisfy the Lagrange--d'Alembert principle:
 $$
 \left(\frac{d}{dt}\frac{\partial L}{\partial \dot q}-\frac{\partial L}{\partial q}\right)\cdot \delta q=0
 $$ for every variation $\delta q\in \tau$. 

Now modify the Lagrangian $L$ by introducing a small parameter $\nu$:
$L_\nu=L+(1/2\nu) g(\dot q^\perp, \dot q^\perp)$, where  $\dot q^\perp$ is the component of the velocity vector $\dot \gamma(t)$ transversal to $\tau$.
For $\nu\not=0$ this is a {\it holonomic system} on $M$ with a nondegenerate metric.
Trajectories of the corresponding system are extremals of the functional $\mathcal L_\nu(\gamma)$, i.e., solutions of the Euler-Lagrange equation corresponding to $\delta \mathcal L_\nu/\delta\gamma=0$: 
$$
\frac{d}{dt}\frac{\partial L_\nu}{\partial \dot q}-\frac{\partial L_\nu}{\partial q}=0\,.
$$
It is well known that nonholonomic systems can be realized as limits of systems subject to friction forces~\cite{caratheodory1933schlitten}. 
Let us encode these forces via a Rayleigh dissipation function $
\frac{1}{\alpha} R(\dot{q}),
$
so that the limit \( \alpha \to +0 \) corresponds to infinitely strong friction.
The corresponding equation assumes the form
$$
\frac{d}{dt}\frac{\partial L_\nu}{\partial \dot q}-\frac{\partial L_\nu}{\partial q}=-\frac{1}{\alpha} \frac{\partial R(\dot{q})}{\partial \dot q}\,.
$$

The two types of systems --- vakonomic and nonholonomic --- can be obtained by taking different limits of the model with dissipation. 
The limit \( \nu \to +0 \) leads to vakonomic dynamics, while the limit \( \alpha \to +0 \) results in nonholonomic dynamics governed by the Lagrange--d'Alembert principle. 
Kozlov~\cite{kozlov1983realization} showed that in the double limit \( \nu, \alpha \to +0 \), with fixed ratio \( \mu = \nu/\alpha \), one obtains a one-parameter family of systems whose equations interpolate between vakonomic (\( \mu \to +0 \)) and nonholonomic (\( \mu \to +\infty \)) regimes.

\bigskip

\subsection{Example: the ideal skate problem}
\label{sec:skate}

As an illustration of the above, we consider a skate moving on an inclined plane. Let \( (x, y) \) denote the contact point of the skate and \( \theta \) its orientation angle. The configuration space is \( Q = \mathbb{R}^2 \times S^1 \).
The skate is subject to the nonholonomic constraint:
$$
\phi := \dot{x} \sin\theta - \dot{y} \cos\theta = 0.
$$
Following \cite{kozlov1983realization, kozlov1992problem}, we consider an extended Lagrangian \( L_\nu \) of an unconstrained system with dissipation given by the Rayleigh dissipation function \( R_\alpha \):
\begin{align*}
    L_\nu &= \frac{1}{2}(\dot{x}^2 + \dot{y}^2 + \dot{\theta}^2) - gx + \frac{1}{2\nu} \phi^2, 
 \qquad
    R_\alpha = \frac{1}{2\alpha} \phi^2.
\end{align*}
Here $g$ is a parameter corresponding to the strength of the gravitational force. 
The corresponding canonical momenta are:
\begin{align*}
    p_x &= \dot{x} + \frac{1}{\nu} \phi \sin\theta, \quad
    p_y = \dot{y} - \frac{1}{\nu} \phi \cos\theta, \quad
    p_\theta = \dot{\theta}.
\end{align*}
The Euler--Lagrange equations with Rayleigh dissipation become:
\begin{align*}
    \dot{p}_x + \frac{1}{\alpha} \phi \sin\theta = -g, \quad
    \dot{p}_y  - \frac{1}{\alpha} \phi \cos\theta = 0, \quad
    \dot{p}_\theta = \frac{1}{\nu} \phi \rho,
\end{align*}
where we define the velocity along the blade direction: 
$$
    \rho := \dot{x} \cos\theta + \dot{y} \sin\theta.
$$

\begin{figure}[ht!]
\centering

\begin{subfigure}[t]{0.48\textwidth}
  \centering
  \includegraphics[width=\textwidth]{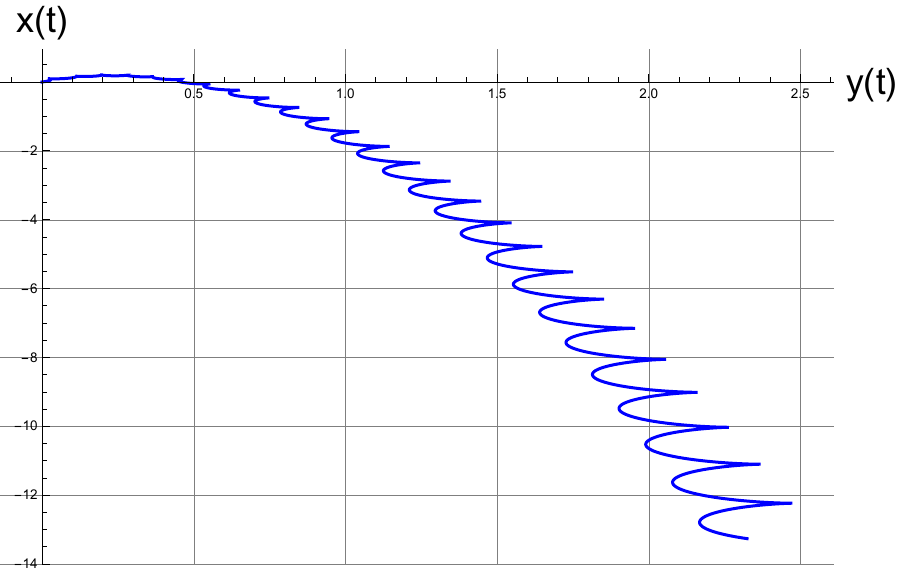}
  \caption{With gravity ($g =1$)}
  \label{fig:trit2a}
\end{subfigure}
\hfill
\begin{subfigure}[t]{0.48\textwidth}
  \centering
  \includegraphics[width=\textwidth]{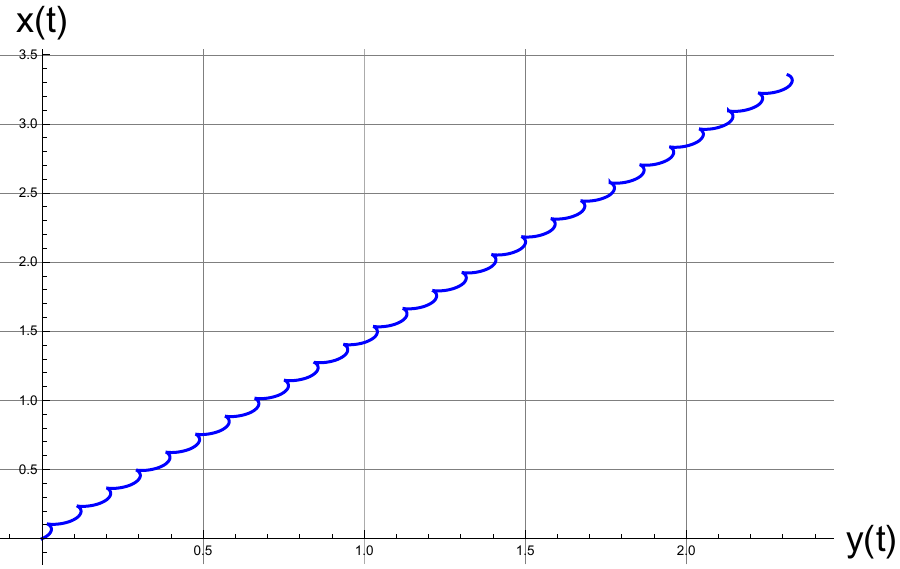}
  \caption{Zero gravity ($g = 0$)}
  \label{fig:trit2b}
\end{subfigure}

\caption{Skate motion according to the vakonomic equations corresponding to $\mu=0$. Initial conditions are: $x_0=y_0=0$, $\theta_0=\pi/4$, $v_0=1$, $\omega_0=-10$. Left: the trajectory of the center of mass of the skate under gravity ($g =1$) for $0 \leq t \leq 8$. Right: the same motion in the absence of gravity ($g = 0$). (Note that the scale of the two panels is different. The motion on the left starts heading up before changing to the  downward drift.)}
\label{fig:trit2}
\end{figure}

\begin{figure}[ht!]
\centering

\begin{subfigure}[t]{0.68\textwidth}
  \centering
  \includegraphics[width=\textwidth]{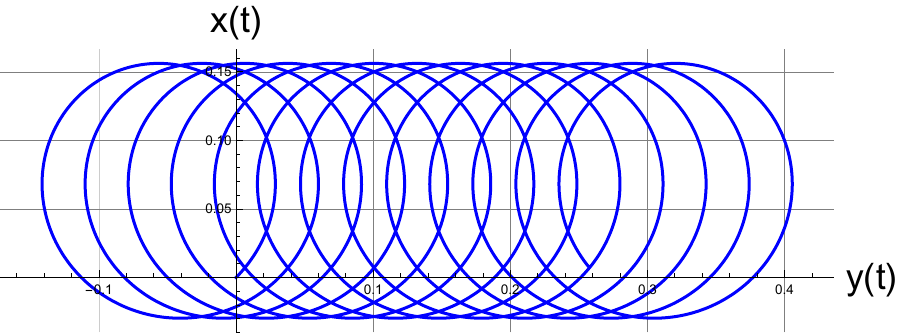}
  \caption{With gravity ($g = 1$)}
  \label{fig:trit1a}
\end{subfigure}
\hfill
\begin{subfigure}[t]{0.28\textwidth}
  \centering
  \includegraphics[width=\textwidth]{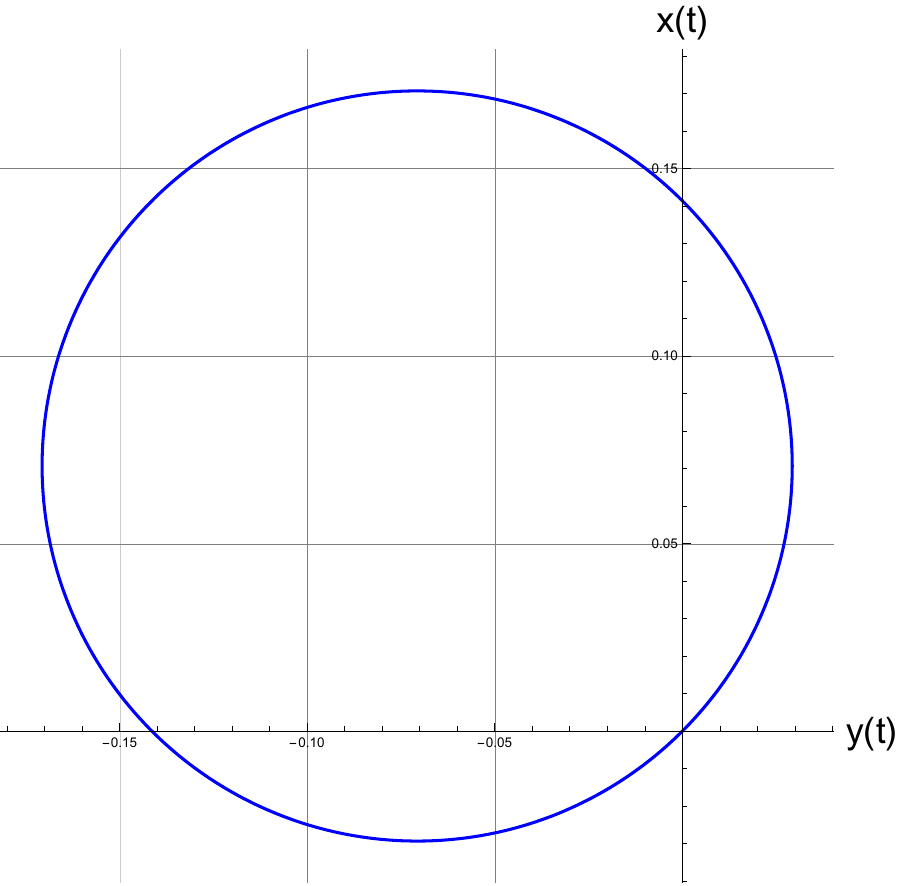}
  \caption{Zero gravity ($g = 0$)}
  \label{fig:trit1b}
\end{subfigure}

\caption{Skate motion governed by the Lagrange--d'Alembert equations in the limit $\mu \to +\infty$. Initial conditions are the same as in Figure~\ref{fig:trit2}. Left: for gravitational acceleration $g = 1$, the trajectory exhibits bounded oscillations forming a cycloidal path~\cite{kozlov1983realization}. Right: in the absence of gravity ($g = 0$), the skate follows a circular trajectory without gravitational drift.}
\label{fig:trit1}
\end{figure}

Now, following \cite{kozlov1983realization}, we take the limit 
\( \nu \to 0 \) and \( \alpha \to 0 \), keeping their ratio fixed \( \mu = \nu/\alpha = \text{const} \). Assuming all fields and initial conditions are \( O(1) \), we obtain:
\begin{align*}
    &\ddot{x} - (\lambda \sin\theta)_t - \mu \lambda \sin\theta = -g, \\
    &\ddot{y} + (\lambda \cos\theta)_t + \mu \lambda \cos\theta = 0, \\
    &\ddot{\theta} = -\lambda \rho, \\
    &\dot{\rho} = -\cos\theta + \lambda \dot{\theta}, \\
    &\dot{\lambda} = -\rho \dot{\theta} + \sin\theta - \mu \lambda.
\end{align*}
This is a dynamical system characterized by an additional parameter \( \mu \). It is easy to show that the energy of the skate given by 
\begin{align*}
    E = \frac{1}{2}(\dot{x}^2+\dot{y}^2+\dot{\theta}^2) +gx = \frac{1}{2}(\rho^2+\dot{\theta}^2) +gx
\end{align*}
is conserved for all values of $\mu$ (we used the constraint in $\dot{x}^2+\dot{y}^2=\rho^2+\phi^2=\rho^2$). 
In the limit \( \mu \to +0 \), we recover the vakonomic equations; see Figure~\ref{fig:trit2}. 

For the opposite limit \( \mu \to +\infty \), we substitute \( \lambda = f/\mu + \mbox{o}(\mu^{-1}) \) and obtain the Lagrange--d'Alembert equations; see Figure~\ref{fig:trit1}. Therefore, the system extended by the parameter \( 0 < \mu < \infty \) interpolates between the vakonomic and Lagrange--d'Alembert equations. Equations for any value of \( \mu \) can be realized in physical systems \cite{kozlov1983realization}; see, for instance, a typical trajectory for an intermediate case in Figure~\ref{fig:trit3}.

The behavior of the skate trajectory is quite rich and strongly depends on the initial conditions. The clear difference between vakonomic and Lagrange--d'Alembert dynamics can be seen from the examples illustrated in Figures \ref{fig:trit2} and \ref{fig:trit1}.

\begin{figure}[ht!]
\centering

\begin{subfigure}[t]{0.68\textwidth}
  \centering
  \includegraphics[width=\textwidth]{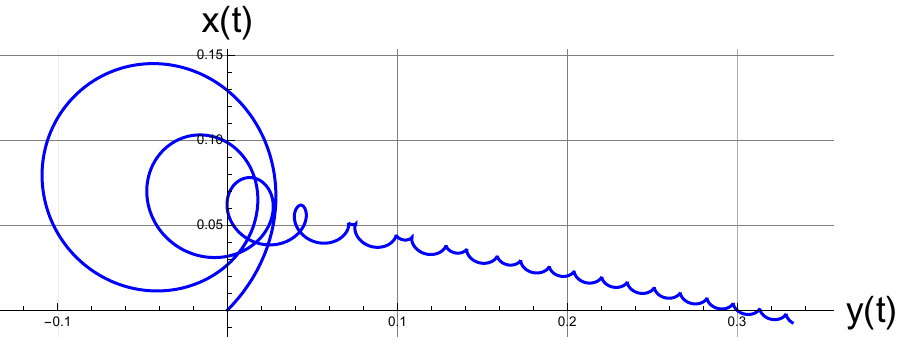}
  \caption{With gravity ($g = 1$)}
  \label{fig:trit3a}
\end{subfigure}
\hfill
\begin{subfigure}[t]{0.28\textwidth}
  \centering
  \includegraphics[width=\textwidth]{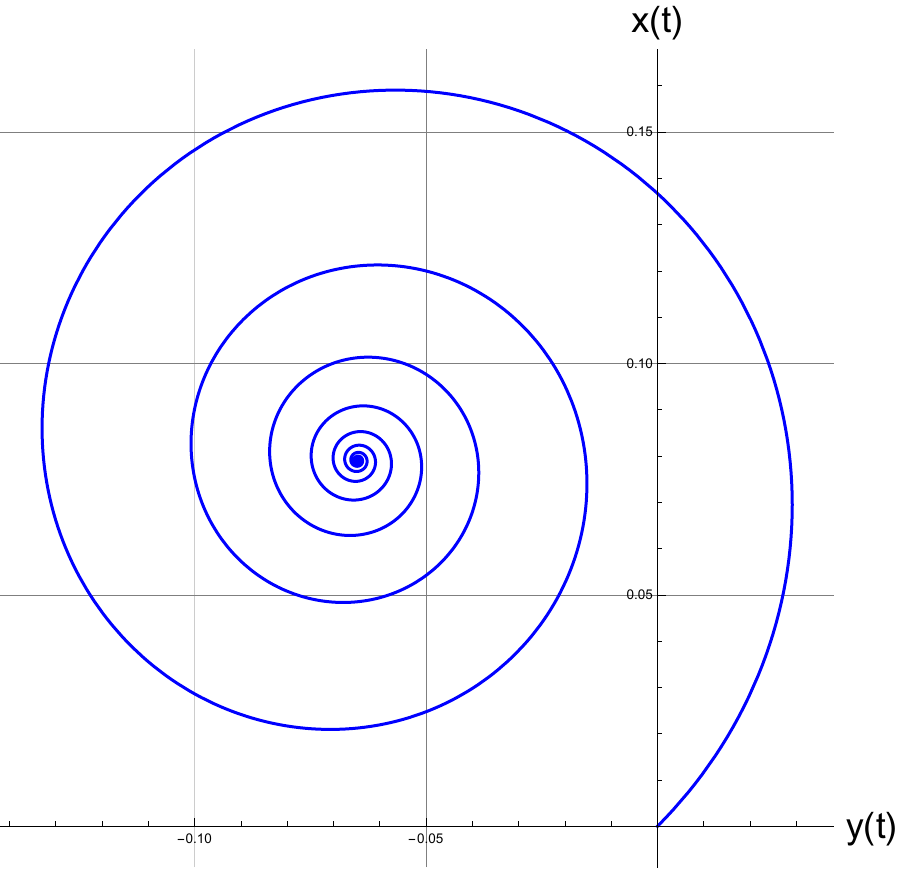}
  \caption{Zero gravity ($g = 0$)}
  \label{fig:trit3b}
\end{subfigure}

\caption{Skate motion for an intermediate value of the parameter \( \mu = 100 \). Initial conditions are the same as in Figure~\ref{fig:trit2}, but the plot scale is different. Left: motion under gravity ($g = 1$). Right: the same dynamics in the absence of gravity ($g = 0$).}
\label{fig:trit3}
\end{figure}

\section{Group symmetry in nonholonomic systems}

\subsection{Subriemannian and Euler-Poincar\'e-Suslov systems}\label{sect:Euler--Arnold}


An important source of examples  is provided by one-sided invariant subriemannian metrics on Lie groups
and the corresponding Euler-Poincar\'e (or Euler--Arnold) equations. As discussed above, 
there are two approaches to define geodesic lines among admissible paths: as ``straightest'' lines, 
defined by the Lagrange-d'Alembert principle,  and as ``shortest'' lines, defined by the variational principle. 

\begin{definition}
The  (``classical'') {\it Euler--Arnold equation} describing geodesics with respect to a right-invariant metric on a Lie group $G$ has the form
$$
 m_t= -{\rm ad}^*_{A^{-1}m}m
 $$
 for a point $m\in \mathfrak{g}^*$ in the dual space to the corresponding Lie algebra $\mathfrak{g}$. Here 
 $A:\mathfrak{g}\to\mathfrak{g}^*$ is an inertia operator  defining the metric on the group $G$
 by fixing the inner product $( v,v) :=\frac 12 \langle v,Av\rangle$ on  $\mathfrak{g}=T_eG$. 
 \end{definition}

This equation is Hamiltonian on 
 $\mathfrak{g}^*$ with the Hamiltonian function $H(m):=-\frac 12 \langle A^{-1}m, m\rangle$ with respect to the Lie-Poisson bracket on $\mathfrak{g}^*$. Note that to write the Euler--Arnold equations in the Hamiltonian form we only need the inverse operator $B:=A^{-1}: \mathfrak{g}^*\to\mathfrak{g}$. 
\medskip
 
Suppose now that we are also given a constraint in the form of a right-invariant distribution $\tau$ on $G$, defined as right shifts of a subspace $\ell\subset \mathfrak{g}$ at the identity $\mathfrak{g}=T_eG$. We also assume that the distribution is nonintegrable (i.e., $\ell$ is not a subalgebra) and bracket-generating on $G$ (i.e., $\ell$ generates the Lie algebra $\mathfrak{g}$ by commutators). This subspace can be defined as the null subspace for 
several elements in the dual space: $\ell:=\mathfrak{g}_a:=\{v\in \mathfrak{g}~|~a_i(v)=0\,\text{ for }\, a_i\in \mathfrak{g}^*,\,\,
 i=1,..,k\}$. Fixing a subriemannian metric on $\tau$, i.e., an inner product on the subspace $\ell$ is equivalent to defining a degenerate operator $B: \mathfrak{g}^*\to\mathfrak{g}$ whose image is $B(\mathfrak{g}^*)=\ell$.

Now we can describe the corresponding vakonomic and  Lagrange-d'Alembert trajectories corresponding to the 
kinetic energy in both settings.
In the {\it vakonomic setting}, we are describing normal subriemannian geodesics, the  ``shortest lines''. They are given by the same Hamiltonian equation on $\mathfrak{g}^*\ni m$, as the classical Euler--Arnold case:
$$
 m_t= -{\rm ad}^*_{B(m)}m\,,
 $$
 but where the operator $B: \mathfrak{g}^*\to\mathfrak{g}$ is non-invertible and $B(\mathfrak{g}^*)=\ell$.
 Its level sets are degenerate quadrics (``cylinders'') in $\mathfrak{g}^*$. The $B$-image in $\mathfrak{g}$ 
  of initial conditions $m$  with different linear combinations of $a_i$ give the same initial velocity:
  $v:=B(m)=B(m+\lambda_ia_i)\in \ell\subset \mathfrak{g}$. However the corresponding trajectories in the group $G$
 with the same initial velocity $v$ may differ, so that the values $\lambda_i a_i$ can be thought of as ``accelerations'' for those trajectories.
 
 \medskip

In the setting of the {\it Lagrange-d'Alembert principle} the corresponding trajectories are governed 
by the so-called {\it Euler-Poincar\'e-Suslov systems}. The corresponding equation has the form similar to the 
Euler--Arnold, but it is non-Hamiltonian in general: it features additional  terms in the right-hand side 
corresponding to the constraints on  $\mathfrak{g}_a:=\{v\in \mathfrak{g}~|~a_i(v)=0\}$ 
for some fixed elements $ a_i\in  \mathfrak{g}^*$. Namely, the Euler-Poincar\'e-Suslov equation 
in that case is
$$
    m_t= -{\rm ad}^*_{A^{-1}m}m+\sum\lambda_i a_i\,,
$$
where $\lambda_i$ are Lagrange multipliers.
This equation is usually written on the Lie algebra itself, where a {\it nondegenerate} operator $A:\mathfrak{g}\to \mathfrak{g}^*$ identifies $\mathfrak{g}$ and $\mathfrak{g}^*$, and  the Lagrange multipliers are determined by the relations $a_i(v)=0$ for $v=A^{-1}m$.

\begin{remark}\label{rem:coincide}
There is a particularly interesting case, when the subspace $\ell\subset \mathfrak{g}$
is itself invariant under the Euler--Arnold equation. In this case the Lagrange multipliers vanish, and 
the corresponding three problems have the same flows: the nonholonomic Lagrange-d'Alembert flow, the vakonomic (or subriemannian geodesic) flow, and the unconstrained (Euler--Arnold) geodesic flow on $G$ restricted to the initial conditions in this subspace $\ell$, so this is a quasiholonomic system, see \cite{jovanovic2001geometry,fedorov2006integrable}.
\end{remark}

\subsection{Heisenberg chain equations on loop groups}
\label{sec:heisenberg}

  The {\it Heisenberg magnetic chain} (or inviscid Landau--Lifschitz) equation has the form
 $$
 \partial_t L = L\times L''\,.
 $$   
It has several equivalent formulations, and, in particular, it is equivalent to the binormal equation 
    $$ 
   \partial_t \gamma =     \gamma'\times\gamma''\,, 
  $$ 
on an arc-length parametrized closed curve $\gamma\subset\mathbb R^3$ under the Gauss map $L=\gamma'$.
The following proposition has been a folklore statement, see e.g. \cite{AK}. 

\begin{proposition}
The Heisenberg chain equation is a geodesic equation for a left-invariant
  subriemannian metric on the loop group   $L{\rm SO}(3)=C^\infty(S^1,{\rm SO}(3))$. 
\end{proposition}

  So, from this point of view, it satisfies the vakonomic principle. 
  Such geodesics are described by a Hamiltonian 
  system on the cotangent bundle $T^*L{\rm SO}(3)$, and hence, due to the invariance, as the Euler--Arnold 
  Hamiltonian equation on  the dual of the Lie algebra   $L\mathfrak{so}(3)=C^\infty(S^1,\mathfrak{so}(3))$. 
  This Hamiltonian formulation  is as follows. Let us identify the Lie algebra 
  $L\mathfrak{so}(3)$ with (the smooth part of) its dual $L\mathfrak{so}(3)^*$   via the pairing 
  $$ 
    \langle X,Y\rangle =-\int_{S^1}{\rm tr}(X(\theta)Y(\theta))\, d\theta\,. 
  $$ 
  While usually for the Euler--Arnold equations one specifies an invertible inertial operator $A:\mathfrak{g}\to\mathfrak{g}^*$, now we define the following noninvertible   self-adjoint operator
  $B: L\mathfrak{so}(3)^*\to L\mathfrak{so}(3)$ acting in the opposite 
  direction: $ B(Y)=-Y''$. 
(If $B$ were invertible, it would have the meaning of the inverse of
the corresponding inertia operator: $B=A^{-1}$.)
  The corresponding Hamiltonian function on the dual space 
  $L\mathfrak{so}(3)^*$  is given by 
$$ 
H(Y):=\frac12 \langle Y,  B(Y)\rangle=-\frac12 \langle Y, Y''\rangle
                     =\frac12 \langle Y', Y'\rangle=-\frac12 \int_{S^1}{\rm tr}(Y'(\theta))^2\, d\theta\,\,
$$ 
for $Y\in L\mathfrak{so}(3)^*$.

The image of $B$ in  $L\mathfrak{so}(3)$  is the subspace $\ell$ of
$\mathfrak{so}(3)$-valued functions on the circle with zero mean. 
On this hyperplane $\ell\subset L\mathfrak{so}(3)$ 
the operator $B$ can be inverted, and this gives 
rise to the  $H^{-1}$-metric 
$$ 
E(X)=\frac 12\langle \partial_\theta^{-1}X, \partial_\theta^{-1}X\rangle\,,
$$
since it is given by the squared $L^2$-norm of the antiderivative 
$\partial_\theta^{-1}X$ for functions $X$ with zero mean, $X\in \ell\subset L\mathfrak{so}(3)$.

Note that the quadratic form on the subspace $\ell$
does not extend to a left-invariant Riemannian metric on a subgroup
of  $L{\rm SO}(3)$. Indeed, this subspace 
$\ell\subset L\mathfrak{so}(3)$ does
not form a Lie subalgebra: the bracket of two loops with zero mean
does not necessarily have zero mean. 
The subspace  $\ell$ of the tangent
space at the identity $\id\in L{\rm SO}(3)$ generates a left-invariant 
distribution on the group $L{\rm SO}(3)$, and we can extend the 
quadratic form $E(X)$ from $\ell$ to a metric on this distribution.
This provides an example of an infinite-dimensional nonintegrable distribution
on a group with a left-invariant subriemannian metric. 
Normal geodesics for this metric 
are described by the same Hamiltonian picture as for a 
left-invariant Riemannian metric on the group, i.e., by the 
Heisenberg magnetic chain (or Landau--Lifschitz) equation.

A similar Landau--Lifschitz equation    $\partial_t L = [L, L'']$ 
 with the same Hamiltonian $H$  can be defined on  the loops in any   semisimple Lie algebra $\mathfrak g$, where 
  $-{\rm tr}(XY)$ is replaced by the  Killing   form  on  $\mathfrak g$, and $[~,\,]$ stands for the commutator in this loop Lie algebra. 
 
\begin{remark}
In the case of the Heisenberg magnetic chain the Lagrange multiplier $\lambda=0$, as the zero 
 mean constraint holds for geodesics in a nondegenerate metrics as well, as discussed in Remark \ref{rem:coincide}. Indeed, the same Hamiltonian equation 
 on $L\mathfrak{so}(3)^*$ can be obtained from an invertible operator $\tilde B := id+B$, i.e., 
 for $\tilde B (Y ) := Y-Y''$, which defines a nondegenerate left-invariant Riemannian metric on the group 
 $L{\rm SO}(3)$. Indeed, the addition of the identity inertia operator does not change the Hamiltonian
dynamics on the orbits, since the latter operator corresponds to the Killing form, and hence on each 
 coadjoint orbit the new Hamiltonian differs from the old one by a constant. Thus in this case vakonomic and Lagrange-d'Alembert trajectories coincide.
In a sense, the above example, as well as the Camassa-Holm and Burgers-type equations discussed below, 
can be regarded as holonomic systems treated from the vakonomic point of view. 
\end{remark}


\subsection{The general Camassa-Holm equation on the diffeomorphism group}
\label{sec:camassa}

The {\it general Camassa--Holm (CH) equation}
$$
u_t + \kappa u_x - u_{txx} + 3uu_x - 2u_xu_{xx} - uu_{xxx} = 0\,
$$
describes an evolution of the fluid velocity $u=u(x,t)$ according to a shallow water approximation 
on the circle (or in 1D in general). For any real constant $\kappa$ there are several ways to view this equation 
as a version of the Euler--Arnold equation on a certain extension of the Lie group 
${G}={\rm Diff}(S^1)$ of circle diffeomorphisms, see \cite{misiolek1998}.

For $\kappa=0$ one obtains the {\it ``classical'' CH equation}.
It is known that in the latter case one can regard the CH equation as the Euler--Arnold equation
for the right-invariant $H^1$-metric on the group ${G}={\rm Diff}(S^1)$ (or on the Virasoro algebra), 
where the metric at the identity
$id\in {\rm Diff}(S^1)$ is given by the following $H^1$-inner product on $\mathfrak{g}={\rm Vect}(S^1)\ni u$:
$$
\frac 12\int_{S^1} ((u, u)+(u_x, u_x)) \,dx\,.
$$

In order to obtain the CH equation with nonzero parameter $\kappa$ one can consider the shift by a constant 
$u\mapsto u+c$, accompanied by reweighting the metric. (Such a shift, however, is not available for the analysis of vector fields $u$ on the line, as it destroys the decay conditions imposed on the fields. A different approach, described below, resolves this difficulty.) Alternatively, one can consider the general CH equation --- either on the line or on the circle --- defined on a trivial central extension of the diffeomorphism group. To fix ideas, consider an extension of circle diffeomorphisms. We start with the Lie algebra which 
 is the  extension $\widetilde{{\rm Vect}(S^1)} ={\rm Vect}(S^1)\times \mathbb{R}$ given by the commutator 
$$
[(u\partial, a), (v\partial, b)]:=([u\partial,v\partial],  \,\, \int_{S^1} u_xv\, dx),
$$ 
which is the extension of the Lie algebra ${\rm Vect}(S^1)$ of vector fields on the circle by means of the {\it (trivial) 2-cocycle} $c(u,v):=\int_{S^1} u_xv\, dx$. 
There exists a central extension 
$\widetilde{ {\rm Diff}(S^1)}={\rm Diff}(S^1)\times \mathbb{R}$ of the Lie group of circle diffeomorphisms ${\rm Diff}(S^1)$ corresponding to the above extension $\widetilde{{\rm Vect}(S^1)}$ of the Lie algebra ${\rm Vect}(S^1)$.

\begin{proposition}[\cite{misiolek1998, grong2015sub, khesin2024information}]
The Euler--Arnold equation for the geodesic flow in the right-invariant $L^2$-metric on the group 
$\widetilde{ {\rm Diff}(S^1)}$  gives the general CH equation with an arbitrary constant $\kappa\in \mathbb R$.    
\end{proposition}

\medskip

However, more importantly for us here,  the general Camassa-Holm equation can be described as a {\it subriemannian geodesic flow} on the (non-extended!)
group ${\rm Diff}(S^1)$ of circle diffeomorphisms. Namely, following \cite{grong2015sub},
 in the Lie algebra ${\rm Vect}(S^1)$ consider the hyperplane $\ell={\rm Vect}_0(S^1)$ of vector fields with zero mean, i.e., ${\rm Vect}_0(S^1):=\{u(x)\partial~|~\int_{S^1}u(x)\,dx=0\}$. This is not a Lie subalgebra, as the commutator is not closed for such vector fields. 
 Now look at the corresponding right-invariant distribution of hyperplanes 
in  ${\rm Diff}(S^1)$ generated by ${\rm Vect}_0(S^1)\subset {\rm Vect}(S^1)$
at the identity $id\in {\rm Diff}(S^1)$. 
It is nonintegrable  (since ${\rm Vect}_0(S^1)$ is not a Lie subalgebra), and it defines a nonholonomic bracket generating distribution (actually, a  contact structure) on the group ${\rm Diff}(S^1)$. 

Now fix the above $H^1$-metric  on ${\rm Vect}_0(S^1)$ and consider subriemannian geodesics on the group  
${\rm Diff}(S^1)$ with respect to the right-invariant metric on this distribution. The subriemannian geodesics are defined by an initial vector and one more parameter (``acceleration'', as we discussed above), and their equation will be the general Camassa-Holm equation, where $\kappa$ is exactly this extra parameter (see \cite{grong2015sub}). Indeed, one can see that addition of this extra term $\kappa u_x$ in the equation does
not change the condition of zero mean $\int u(x)dx=0$, i.e., it is lying in the kernel of the operator $B$ corresponding to the subriemannian metric.

This delivers one more example of  an infinite-dimensional vakonomic system,  the general Camassa-Holm equation as the subriemannian $H^1$- geodesic on the group ${\rm Diff}(S^1)$. Namely, this group is equipped with the right-invariant contact distribution, given at the identity by the constraint $\int_{S^1}u\,dx=0$. 
\medskip



\section{Parity breaking nonholonomic fluids}
\label{sec:parityfluid}

Hamiltonian structures play an important role in both compressible and incompressible fluid dynamics \cite{zakharov1997hamiltonian,morrison1998hamiltonian}. The equations of a classical fluid are invariant with respect to mirror reflections (often called parity transformations). It is interesting,  however, to consider another type of fluids whose equations can contain  parity-odd (or, parity breaking) terms.
Motivated by parity breaking in ideal two-dimensional fluids, the paper \cite{monteiro2023hamiltonian} considered the fluid dynamics with the stress tensor $T_{ij}$ containing parity breaking terms that are of the first order in spacial derivatives. 
As a simplified version of the general model studied in \cite{monteiro2023hamiltonian} we consider
the following stress and viscosity tensors:
\begin{align}
    T_{ij} &= -p(\rho) \delta_{ij} +\eta_{ijkl}\partial_kv_l,
 \label{stress-100}\\
    \eta_{ijkl} &= -\eta_H(\epsilon_{ik}\delta_{jl}+\epsilon_{jl}\delta_{ik}) +\Gamma_H(\epsilon_{ij}\delta_{kl}-\delta_{ij}\epsilon_{kl}).
 \label{eta-100}
\end{align}
Here $p(\rho)$ is the pressure function of the given compressible isotropic fluid, $\eta_{ijkl}(\rho)$ is a viscosity tensor of the fluid. The kinetic coefficients $\eta_H(\rho), \Gamma_H(\rho)$ are functions of the fluid's density $\rho$ and they describe odd viscosity and odd torque, respectively. 
Note that the terms involving $\eta_H$ and $\Gamma_H$ break parity and time reversal invariance of the corresponding dynamics given by the continuity and  Euler equations
\begin{align}
    \partial_t\rho +\partial_i(\rho v_i) &=0,
 \label{cont-100}\\
    \partial_t v_j +v_i\partial_i v_j &= \frac{1}{\rho} \partial_i T_{ij}\,.
 \label{euler-100}
\end{align}

\begin{proposition}
The above dynamics of parity breaking  fluids is dissipation-less.   
The conserved energy is  $  H =\int  [{\rho v^2}/2 +\varepsilon(\rho) ]\,d^2x$ for  the pressure $p=\rho\varepsilon'(\rho)-\varepsilon(\rho)$ and any 
$\eta_H(\rho)$ and $\Gamma_H(\rho)$  (here  prime $'$ denotes the derivative with respect to $\rho$). 
\end{proposition}

It turns out that for generic $\eta_H(\rho)$ and $\Gamma_H(\rho)$ 
the described fluid dynamics does not have a natural Hamiltonian structure. 
On the other hand, in the presence of
an additional relation $\hat\Gamma:=\Gamma_H - \eta_H +\rho \eta_H'=0$  between kinetic parameters 
the dynamics becomes holonomic and Hamiltonian, see  \cite{monteiro2023hamiltonian}.
\medskip

However,  one can enlarge this system by introducing additional fields 
so that it can be understood as satisfying certain {\it nonholonomic constraints in an extended phase space}. 
In particular, the paper \cite{monteiro2023hamiltonian} described a fluid whose fluid particles possess an intrinsic rotational degree of freedom $\ell$. The evolution of the intrinsic angular momentum $\ell$  is given by 
 \begin{equation}\label{eq:delta_l}
  \partial_t \delta\ell +\partial_i(\delta\ell\, v_i) = -2\hat\Gamma \partial_i v_i - \frac{\mu}{\nu} (\delta \ell),     
 \end{equation}
 where 
  $   \delta\ell = \ell +2\eta_H$. 
The new parameter $\mu>0$ describes the relaxation of the intrinsic angular momentum density $\ell$ of the fluid. This parameter $\mu$ is an analogue of the Rayleigh dissipation parameter discussed in
\cite{kozlov1992problem} and Section~\ref{sec:skate}.

The complete system, in addition to the equation \eqref{eq:delta_l} for $\delta\ell$, includes the analogs of dynamics equations  (\ref{cont-100},\ref{euler-100}) with the stress tensor (\ref{stress-100}) modified by $p\to  p^\ell$ and $\eta\to \eta^\ell$, where
\begin{align}
      p^\ell &= p +\frac{1}{2\nu}(\delta \ell)^2 +\frac{2}{\nu}\hat\Gamma\,\delta \ell,
 \label{pell-200}\\
    \eta^\ell_{ijkl} &= \frac{1}{2}\ell (\delta_{ik}\epsilon_{jl}+\delta_{jl}\epsilon_{ik}) +\Gamma_H(\delta_{ij}\epsilon_{kl}-\epsilon_{ij}\delta_{kl}).
 \label{etaell-200}
\end{align}
The parameter $\nu>0$ describes the coupling of the intrinsic angular momentum  $\ell$ of the fluid to the function of the density $\eta_H(\rho)$. Now the equations can be derived from the Hamiltonian of the fluid 
 $$  H_\nu =\int  \left[\frac{\rho v^2}{2} +\varepsilon(\rho) +\frac{1}{2\nu} (\delta \ell)^2\right]\,d^2x\,
$$ 
and the Rayleigh dissipation function
$$
    R_\mu = \int \left[\frac{\mu}{2\nu}  (\delta\ell)^2\right] \,d^2x \,.
$$
For details on the corresponding Poisson brackets see Ref.~\cite{monteiro2023hamiltonian}.
\smallskip

Consider the limits of the Rayleigh dissipation following \cite{kozlov1992problem}. Namely, having fixed $\mu$ one first considers the limit $\nu\to 0$ and from Equation \eqref{eq:delta_l} one observes that 
 $$
 \delta \ell = -\frac{4\nu}{\mu}\hat\Gamma\,\partial_i v_i  +O(\nu^{2})\to 0\,.
 $$
Substituting into (\ref{pell-200},\ref{etaell-200}) and 
keeping only terms of the zeroth order in $\nu$ we obtain equations (\ref{cont-100},\ref{euler-100}) with the stress and viscosity tensors (\ref{stress-100},\ref{eta-100}), where the pressure function changes as follows
\begin{align}
p(\rho) \to p(\rho) -\frac{8}{\mu}\hat\Gamma \,\partial_iv_i\,.
 \label{pressure-400}
\end{align}

Now in the limit $\mu \to +\infty$ the last term in \eqref{pressure-400}) vanishes and we obtain nonholonomic fluid whose dynamics is given by Equations (\ref{cont-100},\ref{euler-100},\ref{stress-100},\ref{eta-100}) and governed by the Lagrange--d'Alembert principle.
One can see that this limiting procedure is analogous to taking the limits $\nu\to 0$ and $\mu\to +\infty$ 
in the skate example in Section~\ref{sec:skate}, following the original derivation in \cite{kozlov1983realization}. 

On the other hand,  the limit $\mu\to 0$ requires the limit $\partial_i v_i\to 0$ in the formula above. The corresponding fluid is incompressible and is described by the vakonomic principles.
\medskip

Note that for  arbitrary $\mu$ and $\nu$ the rate of change of the energy is given by
$\partial_t H_\nu=- 2R_\mu$. In the limit $\nu\to 0$ the dissipation vanishes. The above considerations are summarized in the following statement.

\begin{theorem}{\rm (cf. \cite{monteiro2023hamiltonian})}
The fluid dynamics (\ref{cont-100},\ref{euler-100}) with stress and viscosity tensors given by (\ref{stress-100},\ref{eta-100}) with  pressure function \eqref{pell-200} conserves energy $  H =\int  [{\rho v^2}/2 +\varepsilon(\rho) ]\,d^2x$ for all values of $\mu>0$.

In the limit $\mu \to +\infty$ the system describes a nonholonomic barotropic-type fluid given by Equations (\ref{cont-100},\ref{euler-100},\ref{stress-100},\ref{eta-100}) and governed by the Lagrange--d'Alembert principle.

In the limit $\mu\to 0$ the corresponding fluid is incompressible and is described by the vakonomic principles.
\end{theorem}

One of intriguing open problems is to observe such  nonholonomic fluids in nature. The discussed nonholonomic fluid dynamics does not violate any physics laws and should be possible to realize either in experiments or as an effective theory. For example, the odd viscosity can be realized using time-modulated drive, see  \cite{PhysRevE.101.052606}.


\section{Flows tangent to nonholonomic distributions}

\subsection{Nonholonomic Moser's theorem}\label{sec:Moser}

One of  very suggestive areas of applications of infinite-dimensional nonholonomic dynamics might be 
related to the following nonholonomic version of the classical Moser theorem. Consider a non-integrable distribution $\tau$ on a compact manifold $M$ (we assume no boundary here, although there is a version with boundary as well). While any rolling or skating condition is related to such a setting where we are looking for a horizontal trajectory for this distribution, now we would like to move densities by flows of diffeomorphisms, whose vector fields are subordinated to  $\tau$. 
The motivation for considering densities (or volume forms) in a space with distribution can be related to problems with many tiny rolling balls (e.g. packaging homeopathic pills). It is more convenient to consider the density dynamics of such balls, rather than look at their trajectories individually.

\begin{theorem}[\cite{khesin2009nonholonomic}]\label{thm:nonhol} 
Let $\tau$ be a bracket-generating distribution, and $\mu_0$ and $\mu_1$ be two volume forms on $M$ 
with the same total volume: $\int_M \mu_0=\int_M\mu_1$. Then there exists a diffeomorphism $\varphi$ of $M$ 
which is the time-one-map of the flow $\varphi_t$ of a nonautonomous vector field $V_t$ tangent 
to the distribution $\tau$ everywhere on $M$ for every $t \in [0,1]$, such that $\varphi^*\mu_1 = \mu_0$. 
\end{theorem}

Thus the existence of the ``nonholonomic isotopy'' $\varphi_t$ is guaranteed by the only condition on equality of total volumes for $\mu_0$ and $\mu_1$, just like in the classical case of Moser's theorem without constraints.

\begin{remark}
One of most common proofs of Moser's theorem is based on  the classical Helmholtz-Hodge 
decomposition: any vector field $W$ on a Riemannian manifold $M$ can be uniquely decomposed into the sum
$W =  V +  U$ of $L^2$-orthogonal terms, where $  V= \nabla f$ and ${\rm div}_\mu  U = 0$. Indeed, one can move the density in a required way by using only the gradient part (and one obtains an elliptic equation on its potential). It turns out, there is the
 nonholonomic Hodge decomposition of vector fields on a manifold
with a bracket-generating distribution $\tau$, 
where gradient part $V=\nabla f$ is replaced by the projection $\bar V=P^\tau\nabla f$
of gradients  to the distribution $\tau$ (thus obtaining a hypoelliptic equation with sub-Laplacian ${\rm div}_\mu (P^\tau \nabla f)$ on the corresponding potential $f$), see \cite{khesin2009nonholonomic}.
\end{remark}

In order to describe how it is related to infinite-dimensional geometry, we recall the standard setting of optimal control. Let ${\rm Diff}(M)$ be the group 
of all (orientation-preserving) diffeomorphisms of a manifold $M$. Its Lie algebra ${\rm Vect}(M)$
 consists of all smooth vector fields on $M$. 
Fix a volume form $\mu$ of total volume 1 on $M$. Denote by
${\rm Diff}_\mu(M)$ the subgroup of volume-preserving diffeomorphisms, i.e., the diffeomorphisms
preserving the volume form $\mu$. The corresponding Lie algebra ${\rm Vect}_\mu(M)$ is
the space of divergence free vector fields.

Let $\mathcal W$ be the set of all smooth normalized volume forms in $M$, which is
called the (smooth) Wasserstein space. Consider the projection map $\pi: {\rm Diff}(M)\to \mathcal W$
defined by the pushforward of the fixed volume form $\mu$ by the
diffeomorphism $\varphi$, i.e., $\pi(\varphi)=\varphi_*\mu$. The projection  $\pi: {\rm Diff}(M)\to \mathcal W$ defines
a natural structure of a principal bundle on ${\rm Diff}(M)$ whose structure group is the
subgroup ${\rm Diff}_\mu(M)$ of volume-preserving diffeomorphisms  and fibers $F$ are
right cosets for this subgroup in ${\rm Diff}(M)$. Two diffeomorphisms $\varphi$ and $\tilde \varphi$ lie in
the same fiber if they differ by a composition (on the right) with a volume-preserving
diffeomorphism: $\tilde \varphi = \varphi\circ s, \, s\in {\rm Diff}_\mu(M)$.
On the group ${\rm Diff} (M)$ we define two vector bundles $ver$ and $hor$ whose spaces
at any diffeomorphism $\varphi\in {\rm Diff}(M)$ consist of right translated to $\varphi$ divergence-free fields
and gradient fields
respectively. Note that the bundle $ver$ is defined by the fixed volume form $\mu$, while $hor$ requires a Riemannian metric. Here the bundle $ver$ of translated divergence-free fields is the
bundle of vertical spaces $T_\varphi F$ for the fibration $\pi: {\rm Diff}(M)\to \mathcal W$, while the bundle $hor$
defines a horizontal distribution for this fibration $\pi$.

\begin{remark}
In these terms, the classical Moser theorem  can be thought of as the
existence of path-lifting property for the principal bundle $\pi: {\rm Diff}(M)\to \mathcal W$: any
deformation of volume forms can be traced by the corresponding flow, i.e.,
a path on the diffeomorphism group, projected to the deformation of forms. 
\end{remark}

\medskip

Now let $\tau$ be a bracket-generating distribution on the manifold $M$. Consider the
right-invariant distribution $\mathcal T$ on the diffeomorphism group ${\rm Diff}(M)$  defined at the
identity $id \in {\rm Diff}(M)$ of the group by the subspace  ${\mathcal T}_{id}\subset {\rm Vect}(M) $ of all those vector fields
which are tangent to the distribution $\tau$ everywhere on $M$:
$$
{\mathcal T}_\varphi = \{V \circ \varphi~|~ V (x)\in \tau_x \,{\text{ for all }}\,\,  x \in M\}.
$$

\begin{theorem}\label{thm:path}
    The infinite-dimensional distribution $\mathcal T$ is a nonintegrable
distribution in $ {\rm Diff}(M)$. Horizontal paths in this distribution are flows of
nonautonomous vector fields tangent to the distribution $\tau$  on manifold $M$.
The projection map $\pi: {\rm Diff}(M)\to \mathcal W$  in the presence
of the distribution $\mathcal T$ on ${\rm Diff}(M)$ admits the path-lifting property.
\end{theorem}

\begin{proof}
To see that this  distribution $\mathcal T$ is nonintegrable we consider two horizontal
vector fields $V$ and $W$ on $M$ and the corresponding right-invariant
vector fields $\widetilde V$ and $\widetilde W$ on ${\rm Diff}(M)$. Then their bracket 
at the identity of the group is (minus) their commutator as vector fields $V$ and $W$ in $M$. This commutator
does not belong to the subspace ${\mathcal T}_{id}$, since the distribution $\tau$ is nonintegrable,
and hence at least somewhere on $M$ the commutator of horizontal fields $V$
and $W$ is not horizontal. The second statement immediately follows from the definition of the distribution ${\mathcal T}$.

The path-lifting property for the projection map $\pi: {\rm Diff}(M)\to \mathcal W$  in the presence
of the distribution $\mathcal T$ on ${\rm Diff}(M)$ is 
a restatement of the nonholonomic Moser theorem.  Namely, for a curve
$\{\mu_t~|~\mu_0 =\mu\}$ in the space $\mathcal W$ of smooth densities Theorem~\ref{thm:nonhol} proves that
there is a curve $\{\varphi_t~|~\varphi_0 = id\}$ in  ${\rm Diff}(M)$, everywhere tangent to the distribution
$\mathcal T$ and projecting to $\{\mu_t\}: \pi(\varphi_t) = \mu_t$.    
\end{proof}

Subriemannian 
geodesics  (which are vakonomic systems with purely kinetic Lagrangians) in the group ${\rm Diff}(M)$ subordinated to the infinite-dimensional distribution  $\mathcal T$
are discussed in \cite{khesin2009nonholonomic, agrachev2009optimal}. Of particular importance are horizontal geodesics, which are allowing fastest moves of densities, while their flows are tangent to a given distribution $\tau$ on $M$. 
More on subriemannian structures on groups of diffeomorphisms,  examples of normal and  abnormal geodesics in that infinite-dimensional context, and a subriemannian version of the Euler--Arnold equation can be found in 
\cite{Arguillre2014}.

Here are two examples of possible applications of the above theory. 
\medskip

\subsection{Example: Transmission flows  in the visual cortex}
\label{sec:visualcortex}

It is now widely accepted, possibly after remarkable paper \cite{hoffman1989visual}, 
that the visual cortex can be regarded as a contact bundle. Indeed, 
the sensory cortex of the
brain is arranged in a structure that is simultaneously ``topographic'' (a pointwise
mapping), layered, and columnar. The microcolumns in the columnar structure
exhibit both a directional and an areal response in addition to the pointwise one. 
These directional-areal response fields are contact elements over the visual
manifold, the ``base'', that generate visual contours as the ``lifts'' of the form stimuli from the
retina into a contact bundle embodied in the visual cortex itself.

In other words, neurons are sensitive not only to the position of an observed object, but also to the
direction of its contour on the retina surface. Thus a rough approximation of the {\it visual cortex}  can regard it {\it as a space of contact elements}.
The latter space is 3-dimensional: 2 dimensions for the position on the retina surface, and one for the observed direction, an element of $S^1$. It has a natural contact structure, given by the skate condition: the contact planes are spanned by the two fields, namely, by the field rotating the direction of the contact element about its tangency point and by the field moving the point of contact along the element direction, see e.g. \cite{arnold1989mathematical}.

The signal in the cortex 
is transmitted in the fastest way along the horizontal curves for this 2D contact distribution. 
Hence the usefulness of the (finite-dimensional) contact geometry in neuroscience.

Now notice that in order to transmit not just separate points but a whole visual picture, it is best to describe 
the {\it evolution of the signal density} along this contact distribution. This is exactly the setting of the nonholonomic Moser theorem and nonholonomic optimal transport \cite{khesin2009nonholonomic, agrachev2009optimal}: one is looking for a (possibly faster) way to transport a density of signals by diffeomorphisms whose flow is tangent to the contact distribution of the visual cortex. Furthermore, by adding colors, brightness,  and other parameters to the visual signal one can set a similar transport problem for nonholonomic distributions in higher dimensions.

\begin{remark}
    Theorems \ref{thm:nonhol} and \ref{thm:path} allow one to generalize  the dynamics of the signal density
    in the visual cortex from contact to arbitrary bracket-generating distributions. The latter might need more commutators to generate the whole tanget space at certain regions and hence have a slower pace of transmitting the signal. This could be particularly important for processing images in such eye regions as scotomas, and in particular in the optic disc, the spot where the optic nerve is exiting the retina: apparently the nonintegrable distributions in the areas of visual cortex corresponding to neurons in scotomas have higher degrees of nonholonomic degeneracy.  
\end{remark}
\medskip

\subsection{Example: Potential flows for the Burgers equation}
\label{sec:burgers}

Return to the setting of the diffeomorphism group ${\rm Diff}(M)$ fibered over the space of densities
${\mathcal{W}}={\rm Dens}(M)$ by the projection $\pi: {\rm Diff}(M)\to \mathcal W$.
On the density space $\mathcal{W}$  there  exists a metric 
inspired by the following  {\it optimal mass
transport problem}: find a (sufficiently regular) map $\eta:M\to M$ that pushes
the  measure $\mu$ forward to $\nu$ and attains the minimum
of the $L^2$-cost functional 
$\int_M \operatorname{dist}^2(x,\eta(x))\mu$ among all such maps, where $\operatorname{dist}$ is the Riemannian distance on $M$.
The minimal cost of transport defines the  
{\it Wasserstein $L^2$-distance} ${\operatorname{Dist}}$ on densities $\operatorname{Dens}(M)$:
$$
{\operatorname{Dist}}^2(\mu, \nu)
:=\inf_\eta\Big\{\int_M \operatorname{dist}^2(x,\eta(x))\mu~
|~\eta_*\mu=\nu\Big\}\,.
$$
This Wasserstein distance function is  
generated by a (weak) Riemannian metric on the space $\mathcal{W}$ of smooth densities.
One can see that, due to the Hodge decomposition,  the most effective way of moving density is by gradient vector fields. 

\medskip

It turns out that there also exists a natural $L^2$ metric on the group ${\rm Diff}(M)$, see \cite{otto2001geometry}. 
Its geodesics are given by solutions to the (inviscid) {\it Burgers equation} $\partial_t u+\nabla_u u=0$
for a vector field $u$ on $M$, where $\nabla_u u$ stands for the covariant derivative on $M$. Solutions of the Burgers equation are time-dependent vector fields on $M$
that describe the following flows of fluid particles: 
each particle moves with constant velocity (defined by the initial condition) 
along a geodesic in $M$.

Geodesics on the density space $\mathcal{W}$, particularly important for optimal transport,  can be obtained from horizontal geodesics on the group ${\rm Diff}(M)$. Horizontal geodesics in  ${\rm Diff}(M)$ correspond to potential solutions of the Burgers equation: their initial conditions are given by {\it gradient fields}: 
$u_0=\grad f_0$. It turns out that then such geodesics remain potential for all times, and the evolution of their potentials is described by the Hamilton-Jacobi equation 
$$
\partial_t f_t+ (\nabla f_t, \nabla f_t)=0\,,
$$ 
see e.g. 
\cite{otto2001geometry, khesin2021geometric}.
\medskip

Here we again observe the phenomenon that the subriemannian geodesics  for the {\it infinite-dimensional nonintegrable horizontal distribution} on the group ${\rm Diff}(M)$ given by right translations of gradient fields on $M$
coincide with Riemannian $L^2$-geodesics  on ${\rm Diff}(M)$ with potential initial conditions. Namely, a Riemannian geodesic that started being tangent to the distribution $hor$, i.e., as a potential Burgers solution,
remains tangent to it for all times, and hence it coincides with a subriemannian geodesic for the same initial condition.

\medskip


\section{Subriemannian approximations of fluid dynamics}
\label{sec:fluidapprox}

There is a beautiful application of subriemannian geodesics in modelling of fluid flows.
The hydrodynamical Euler equations define an infinite-dimensional Hamiltonian system as the Euler--Arnold equation on the group 
${\rm Diff}_\mu(M)$ of volume-preserving diffeomorphisms, cf. Section~\ref{sect:Euler--Arnold}.
Recall that the Euler equations describe the following evolution of the fluid velocity field $v=v(x,t)$:
$$
\partial_t v +\nabla_v v=-\nabla p, \qquad {\rm div}\,v=0\,,
$$
where the pressure function $p$ is defined uniquely modulo an additive constant, $v\in \mathfrak g={\rm Vect}_\mu(M)$ is divergence-free with respect to the volume form $\mu$, and $\nabla_v$ is the covariant derivative with respect to a metric on $M$. 
The Hamiltonian formulation of this equation on the dual space $\mathfrak g^*$ describes 
an evolution of 1-form $v^\flat$ metric-dual to $v$:
$$
\partial_t v^\flat +L_v v^\flat=-d\tilde p\,.
$$
In \cite{Sina, Chern2024} it was proposed to approximate this dynamics by using a {\it non-integrable finite-dimensional distribution} on the same infinite-dimensional group ${\rm Diff}_\mu(M)$. 
(Compare this with the non-integrable  {\it infinite-dimensional distribution} $\mathcal T$ on the diffeomorphism group ${\rm Diff}(M)$ in Section \ref{sec:Moser}.)

Namely, fix a finite-dimensional subspace $\ell\subset {\rm Vect}_\mu(M)$, 
which is not a Lie subalgebra, in the the space of divergence-free vector fields $ {\rm Vect}_\mu(M)$ in a manifold $M$. 
(One may consider, e.g., $\ell$ spanned by a finite number of Fourier harmonics or by vector fields supported near vertices of the regular lattice in $\mathbb R^d$, convenient for computations.) The corresponding right-invariant distribution $\mathcal L$  on the group ${\rm Diff}_\mu(M)$ 
obtained by translating $\ell=\mathcal L_{id}$  to any element of the group is a non-integrable finite-dimensional distribution.
The standard right-invariant $L^2$-metric restricted to this distribution equips $({\rm Diff}(M), \mathcal L)$ with a subriemannian metric.
 
The corresponding subriemannian geodesics (or equivalently, vakonomic trajectories) are described by means of the Hamiltonian flow on the corresponding  dual Lie algebra $\mathfrak g^*= {\rm Vect}^*_\mu(M)$ for the new degenerate quadratic Hamiltonian function $\tilde H$, dual to the $L^2$-inner product on $\ell$, see \cite{Sina}:
$$
\partial_t u +L_v u=-d\tilde p\,,
$$
where  $v\in \ell\subset {\rm Vect}_\mu(M)$, while $u=v^\flat+\zeta$ with $\tilde H(\zeta)=0$, i.e. $u$
is a 1-form differing from $v^\flat$ by an element of the annihilator of $\ell$.
Since this equation is still Hamiltonian on $\mathfrak g^*= {\rm Vect}^*_\mu(M)$ with an adjusted Hamiltonian, it has the same symmetries and Casimirs as the original Euler equation. In particular, the vorticity field is still frozen into the flow, as in the original Kelvin law.
\smallskip

This can be compared with the Lagrange--d'Alembert principle for the same nonholonomic constraints:
$$
\partial_t u^\flat +P(L_v u^\flat) =-d\tilde p\,,
$$
where  $v\in \ell\subset {\rm Vect}_\mu(M)$, while $P$ is a projection on $\mathfrak g^*$ of the coadjoint action operator $L_v$ 
to a certain dual to the subspace $\ell$, see \cite{Sina}.  Unlike the subriemannian setting above, this projection might not preserve the coadjoint orbits, although the dynamics preserves the kinetic energy.

\begin{remark}
One can also approximate the infinite-dimensional group of volume-preserving diffeomorphisms using 
a finite-dimensional Lie group, where associated discrete Euler equations are derived either from 
a variational principle or from the Lagrange--d'Alembert principle with nonholonomic constraints. 
In particular, in \cite{marsden2009structure} there was proposed an approach utilizing an Eulerian, 
finite-dimensional representation of volume-preserving diffeomorphisms that encodes the displacement of a fluid from its initial configuration using special orthogonal signed stochastic matrices. From this particular discretization of the configuration space,  regarded as (a subset of) a finite-dimensional Lie group, one can derive a right-invariant discrete equivalent to the Eulerian velocity through its Lie algebra, i.e., through antisymmetric matrices whose columns sum to zero. 
By  imposing  nonholonomic constraints on the velocity field to allow transfer only between neighbouring cells during each time update the authors of \cite{marsden2009structure} apply the Lagrange--d'Alembert principle to obtain the discrete equations of motion for their fluid representation. The resulting Eulerian Lie-group integrator is structure-preserving, it also manifests good long-term energy behavior and  numerical properties in that approximation.
\end{remark}


\section{Cars with many trailers and snake motions}

\subsection{The \texorpdfstring{$n$}{n}-trailer systems and Goursat distributions}
\label{sec:ntrailers}

The $n$-trailer system in control theory is a  kinematical model for a (vertical) unicycle towing $n$ trailers, see e.g. \cite{Jean1996, PasillasLepineRespondek2001, MONTGOMERY2001459}. In this model the tow hook of each trailer is located at the center of its unique axle, and for simplicity one often assumes that  the distances between any two consecutive trailers are equal. The configuration space of such a system is $\mathcal M^{n+3}=\mathbb R^2\times (S^1)^{n+1}$ equipped with a two-dimensional distribution spanned by the admissible infinitesimal 
motions of the unicycle (or, equivalently, of the skate discussed in Section \ref{sect:skate}).
One can define this distribution inductively: first consider the pair of vector fields $(\tau^0_1, \tau^0_2)$ on $\mathbb R^2\times S^1$ describing  the kinematics of the unicycle towing no trailers: 
$$
\tau^0_1=\frac{\partial}{\partial \theta_0}\quad \text{ and }\quad \tau^0_2=\cos\theta_0 \frac{\partial}{\partial x}+\sin \theta_0 \frac{\partial}{\partial y}\,.
$$
Now the {\it $n$-trailer system} is defined by adding one trailer at a time, which correspond to a sequence of prolongations.
Namely, suppose that a pair of vector fields $\tau^{n-1}:=(\tau^{n-1}_1, \tau^{n-1}_2) $ on
$\mathbb R^2\times (S^1)^{n}$ corresponding to a unicycle with $(n-1)$ trailers was defined. Then the next pair 
of vector fields $\tau^{n}:=(\tau^{n}_1, \tau^{n}_2) $ on
$\mathbb R^2\times (S^1)^{n+1}$ is given by
$$
\tau^n_1=\frac{\partial}{\partial \theta_n}\quad \text{ and }\quad \tau^n_2=\sin(\theta_n-\theta_{n-1})\tau^{n-1}_1+
\cos (\theta_n-\theta_{n-1}) \tau^{n-1}_2\,.
$$
Here $x,y$ give position of the last trailer on $\mathbb R^2$ and $\theta_0,..., \theta_n$ stand for the angles between trailer's axle (starting with the last one) and the $x$-axis.

This 2D distribution is a canonical example of the Goursat distribution, where successive commutator brackets of the vector fields belonging to the derived distribution grow by 1 dimension at a time. Here is the formal definition, see e.g. 
\cite{montgomery2002tour}.

\begin{definition}\label{def:Goursat}
A {\it Goursat distribution} on a manifold $M$ of dimension $n\ge 3$ is a  two-dimensional distribution $\mathcal D$ such
that, for $0\le i\le n-2$, the elements of its derived flag satisfy $\dim \mathcal D^{(i)}(p) = i+2$, for each point $p \in M$.
Recall that the {\it derived flag} of a distribution $\mathcal D$ is the sequence $\mathcal D^{(0)}\subset \mathcal D^{(1)}\subset \dots$ defined inductively for $ i\ge 0$ by 
$$
\mathcal D^{(0)}: = \mathcal D \quad {\rm and}\quad  \mathcal D^{(i+1) }:= \mathcal D^{(i)} + [\mathcal D^{(i)}, \mathcal D^{(i)}].
$$
\end{definition}

The classical theorem of von Weber-Cartan-Goursat claims that any Goursat distribution in $M^n$  at a generic point 
is diffeomorphic to the one spanned by the following pair of vector fields in $\mathbb R^n$:
$$
\mathcal D=(\frac{\partial}{\partial x_n}, \, x_n\frac{\partial}{\partial x_{n-1}}+x_{n-1}\frac{\partial}{\partial x_{n-2}}+\dots
+x_3\frac{\partial}{\partial x_{2}}+\frac{\partial}{\partial x_{1}})\,.
$$
The trailer system $\tau^{n-3}=(\tau^{n-3}_1, \tau^{n-3}_2) $ with $n$-dimensional configuration space in a neighborhood of a typical point can be reduced to this Goursat normal form $\mathcal D$ in $\mathbb R^n$, see e.g. \cite{PasillasLepineRespondek2001}.
\medskip

Another important example of a two-dimensional Goursat distribution is the classical Cartan distribution 
on the $(s+2)$-dimensional space of $s$-jets of functions $f(x)$ in one variable. This distribution can be described by zeros of $s$ differential 1-forms 
$$
\alpha_1 =dy-z_1\,dx, \, \alpha_2=dz_1-z_2\,dx,\, \dots, \alpha_s =dz_{s-1}-z_s\,dx
$$
in the space $(x,y,z_1,\dots, z_s)$,
where $y$ represents the value of $f$ at $x$ and $z_i$ represents the value at $x$ of the $i$th derivative of $f$. 
Then the $s$-jet of any function $y=f(x)$ is tangent to the two-dimensional distribution defined by the intersection 
of zero hyperplanes of the  1-forms $\alpha_1,\dots, \alpha_s$. 


\subsection{Parking a car and a dimensional shift}
\label{sec:parkingcar}

As a side note, it is worth mentioning that the description of a car with trailers, unlike the common perception, bumps up
the dimensions in this problem. Namely,  {\it a car} without a trailer corresponds to a {\it four-dimensional configuration space}, not to the three-dimensional configuration space of {\it a  unicycle or a skate}. Correspondingly, {\it a car with $n$ trailers} 
has  an $(n+4)$-dimensional  domain $\mathcal C^{n+4}\subset \mathcal M^{n+4}=\mathbb R^2\times (S^1)^{n+2}$
as a natural configuration space.

\begin{figure}[ht!]
\begin{center}
\includegraphics[width=0.5\textwidth]{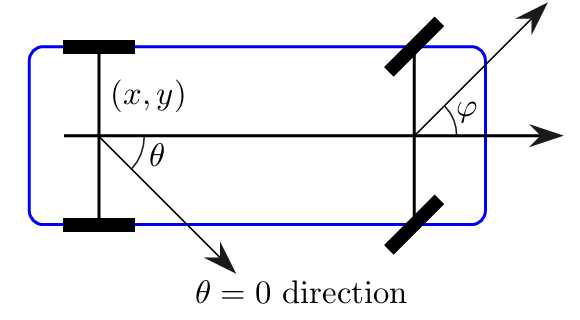}
\end{center}
\caption{The car position is described by  its midpoint $(x,y)$ of the axle, the angle $\theta$ of the car axle with a fixed direction, and the steering angle $\varphi$ of the front wheels, see \cite{michor2008topics}.}
\label{fig:car}
\end{figure}

Indeed, for a car in the plane the configuration space consists of all quadruples 
$(x,y, \theta, \varphi)\in \mathbb R^2\times S^1\times I=:\mathcal C^4$, where $(x,y)$ is the position of the midpoint 
of the rear axle, $\theta$ is the direction of the car axle, and $\varphi$ is the steering angle of the front wheels 
with the range within some interval $I$, e.g. $I=(-\pi/4, \pi/4)$, see Figure \ref{fig:car} and \cite{michor2008topics}. 
We emphasize that the car's configuration space is four-dimensional, as one actually needs four parameters to describe
the corresponding two control vector fields: 
$$
{\rm steer} := \frac{\partial}{\partial\varphi} \, {\text{ and }}\, {\rm drive} := \cos \theta \frac{\partial}{\partial x}
 +\sin\theta  \frac{\partial}{\partial y} +(1/l )\tan \varphi \frac{\partial}{\partial \theta}\,,
 $$
 which together span the distribution $\mathcal D=({\rm steer}, {\rm drive})$. (Here $l$ is the span between the front and rear axles.)
The fields obtained by their commutators 
 $$
 {\rm turn }:=[{\rm steer, drive}] \quad {\rm and} \quad {\rm park}:=[{\rm drive, turn}]
 $$ 
 span  the corresponding distributions of the derived flag: $\mathcal D^{(1)}=({\rm steer}, {\rm drive}, {\rm turn})$ and 
  $\mathcal D^{(2)}=({\rm steer}, {\rm drive}, {\rm turn}, {\rm park})$. 
  Explicitly, one obtains the fields
 $$
 {\rm turn}=h(\varphi)\frac{\partial}{\partial \theta} \, {\text{ and }}\,  
 {\rm park}=h(\varphi)(\sin \theta \frac{\partial}{\partial x}-\cos\theta  \frac{\partial}{\partial y} )
 $$
 for $h(\varphi):=1/(l\cos^2\varphi)$. Note that the field ``turn'' is collinear with the rotation field of the unicycle, while 
 the ``park'' vector field moves the car orthogonally to its axis, i.e., provides the parallel parking.
 
The fact that the  
 whole tangent space  of the configuration space $\mathcal C^4$ is spanned at each point (i.e., the distribution 
 is bracket-generating) implies  that every point of $\mathcal C^4$ is attainable. Its common-sense corollary is that, 
 in principle, one can park a car at any point and  in any direction in the plane. 
 The above consideration can be summarized in the following statement.

\begin{proposition}
The two car control vector fields ${\rm steer}$ and ${\rm drive}$ span the Engel distribution, i.e., a generic Goursat two-dimensional distribution in a four-dimensional space $\mathcal C^4$. 
\end{proposition}

This follows from the explicit expressions for the fields given above. 
Note that the ``turn'' field $h(\varphi){\partial}/{\partial \theta}\in\mathcal D_C^{(1)}$ arises in 
the 3-dimensional commutator  of the first two fields $({\rm steer}, {\rm drive})=:\mathcal D_C$ 
in the car setting in $\mathcal C^4$, while the ``turn'' field $\tau^0_1={\partial}/{\partial \theta}$ appears already in the initial  
2-dimensional distribution $\mathcal D_M:=\tau^0=(\tau^0_1,\tau^0_2)$ in the unicycle setting in $\mathcal M^3$!

In fact, one can see that the forgetful map $\mathcal C^4\to \mathcal M^3$ sending 
$(x,y, \theta, \varphi)\mapsto (x,y, \theta)$ by forgetting the car steering angle $\varphi$ takes the distribution
$\mathcal D_C^{(1)}$ to $\mathcal D_M=\mathcal D_M^{(0)}$.

\bigskip

 Next, for a car with a trailer one has the 5-dimensional configuration space ${\mathcal C}^5$, where the angle with 
 the trailer is added to the list of coordinates, etc.  The controls are still limited to the same two vector fields ``steer'' and ``drive'', 
 which generate,  by taking their iterated commutators, the whole tangent space of 
 ${\mathcal C}^5$. Similarly, the systems of a car with $n$ trailers are all described by generic  Goursat distributions
 in the corresponding configuration spaces ${\mathcal C}^{n+4}$, as was mentioned before. Thus their dimensions are  one larger than those of ${\mathcal M}^{n+3}$ for a unicycle with $n$ trailers.  The car-trailer system along with its distribution $\mathcal D_C$ and its prolongation in ${\mathcal C}^{n+4}$ are projected  by the above map of forgetting the steering angle to the unicycle-trailer system with the Goursat distribution $\mathcal D_M$  and its prolongation in ${\mathcal M}^{n+3}$. In particular the distribution $\mathcal D_C$ is also Goursat.


\subsection{Sleighs with strings, snake motions, and infinite-dimensional Goursat}
\label{sec:snake}

Definition \ref{def:Goursat} is also valid in an infinite-dimensional setting. 

\begin{definition} 
A {\it Goursat structure on an infinite-dimensional manifold} $\mathcal M$  
is a  two-dimensional distribution $\mathcal D$ such that, for all $ i\ge 0$, the elements of its derived flag 
satisfy $\dim \mathcal D^{(i)}(p) = i+2$, for each point $p \in \mathcal M$.
\end{definition}

The main example of the latter is the infinite-dimensional jet space equipped with the Cartan distribution.
Prolongations of functions of one variable (i.e., considered with all their derivatives) are tangent to this distribution. 
\medskip

Above we looked at a car/unicycle/skate with $n$ trailers, and now we send the number of trailers  to infinity. 
A natural limiting infinite-dimensional system has the following kinematic description. At each time moment $t$ it is a smooth unstretchable string described by an embedded arc-parametrized curve $z(s):=(x,y)(s)\in \mathbb R^2$ 
of fixed length. Its evolution $z(s,t) $ is subordinated to the following nonholonomic constraint:
at each moment its time derivative, $\partial_t z$, is collinear with the curve's tangent, $\partial_s z$.
This is {\it the infinite-dimensional skate constraint}, where the motion of a skate is possible only in the current direction of the skate itself, while it cannot move across, i.e., transversally to that direction.
 This constraint implies that the image of the map $z: (s,t)\mapsto z (s,t)\in \mathbb R^2$ is one-dimensional: at any moment $t_*$ the velocities (i.e., time derivatives) of points on the curve
 $z(s, t_*)$ for any $s$ are directed along the curve $z(0, t)$ made by the motion of the curve's own head point. Since we assume that $s$ is an arc-length parameter, the time evolution reduces to the
 shift in the image parametrization, so that $z(s,t)=u(s+f(t))$. This model of  a {\it snake-type motion} is also known as the  {\it Chaplygin sleigh with a string}, cf. \cite{Zenkov}.
  
\begin{theorem}
The snake-type motion with the  infinite-dimensional skate constraint
corresponds to a curve subordinated to an  infinite-dimensional  Goursat structure, the Cartan distribution in the infinite jet space.
\end{theorem}

\begin{proof}
 Indeed, consider the curve $z(s, t_*)$ at any moment $t_*\in (0,1)$. At the moment $t_*+\delta t$ the curve
$z(s, t_*+\delta t)$ has the same prolongation in $s$ at the curve $z(s, t_*)$, since all its derivatives are predefined
by the one-dimensional image in $\mathbb R^2$. Hence the motion in $t$ can be regarded as the motion
along its prolongation of the fixed curve in the plane, while the prolongation must be everywhere tangent to the Cartan distribution in the infinite jet space. Furthermore, a curve in the plane can be locally regarded as
the graph of a function $\mathbb R\to \mathbb R$. 
Thus the statement  follows from one-dimensionality of the image in the plane and the properties of the Cartan distribution for functions of one variable.   
\end{proof}

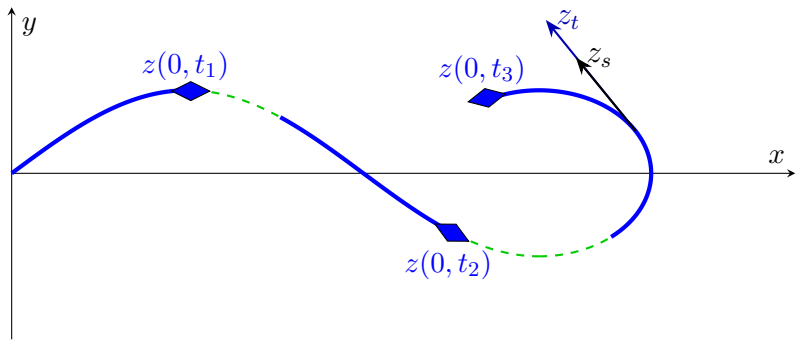
\begin{figure}[ht]
\centering

\begin{tikzpicture}
  \begin{axis}[
    axis lines=middle,
    xmin=0, xmax=7,
    ymin=-2, ymax=2,
    xtick=\empty, ytick=\empty,
    xlabel={$x$}, ylabel={$y$},
    width=12cm, height=6cm,
    samples=300,
    clip=false
  ]

    \addplot [green!80!black, dashed, thick, domain=0:4.7124]
      ({x}, {sin(deg(x))});

    \addplot [green!80!black, dashed, thick, domain=-90:120]
      ({4.7124 + cos(x)}, {sin(x)});

    \addplot [blue, ultra thick, domain=0:1.5708]
      ({x}, {sin(deg(x))});

\begin{scope}[rotate=0]
  \node[
    shape=diamond,
    fill=blue,
    draw=black,
    minimum size=6pt,
    xscale=1,
    yscale=0.5,
    transform shape
  ] at (axis cs:{1.6}, {0.99}) {};
\end{scope}
    
    \node[blue] at (axis cs:1.55, 1.3) {$z(0,t_1)$};

    \addplot [blue, ultra thick, domain=2.4:3.9708]
      ({x}, {sin(deg(x))});

\begin{scope}[rotate=-27]
  \node[
    shape=diamond,
    fill=blue,
    draw=black,
    minimum size=6pt,
    xscale=1,
    yscale=0.5,
    transform shape
  ] at (axis cs:{3.07}, {1.55}) {};
\end{scope}

    \node[blue] at (axis cs:3.9, -1.1) {$z(0,t_2)$};

    \addplot [blue, ultra thick, domain=-50:120]
      ({4.7124 + cos(x)}, {sin(x)});


\begin{scope}[rotate=12]
  \node[
    shape=diamond,
    fill=blue,
    draw=black,
    minimum size=6pt,
    xscale=1,
    yscale=0.5,
    transform shape
  ] at (axis cs:{5.1 + cos(120)}, {sin(120)-1.21}) {};
\end{scope}

\node[blue] at (axis cs:4.2, 1.25) {$z(0,t_3)$};
  \end{axis}


    \draw[->, thick, blue!70!black] (8.3, 2.77) -- (7.1, 4.25) node[right=-0pt] {\large $z_t$};
    
    \draw[->, thick, black!70!black] (8.3, 2.77) -- (7.5, 3.75) node[right=-0pt] {\large $z_s$};
\end{tikzpicture}

\caption{Illustration of the infinite-dimensional ``snake constraint'': the string evolves so that the velocity \( z_t \) of each point remains collinear to its tangent vector \( z_s \), enforcing the nonholonomic skate-like constraint.  The evolving curve slides along itself and follows the trajectory of its own head point \( z(0,t) \).
Blue segments show the shape of the snake at times \( t_1 < t_2 < t_3 \), each tangent to the common trajectory (dashed green). }
\label{fig:snake-geometry} 
\end{figure}

\begin{remark}
There is an interesting corollary\footnote{We are grateful to R.~Montgomery for this observation.} of the above consideration: 
a typical trajectory of the snake-type motion is $C^\infty$-smooth, but not analytic. Indeed, an analytic curve $z(s, t_*)$ is defined 
uniquely by its values at any neighborhood of any point $s_*$, and hence there remains no freedom in the head motion except for different time parametrization of the same trajectory. Essentially this means 
that the control system obtained by taking  a continuous limit of the many-trailer system is naturally to be defined as 
$C^\infty$-map $z(s,t)=u(s+f(t))$ in order to keep its motion flexible, rather than predetermined.
\end{remark}

The above discussion concerns the possible {\it kinematics} of the snake, where the head can follow any immersed curve without restrictions. The corresponding  {\it dynamics} of the balanced Chaplygin sleigh 
 with a string as an infinite-dimensional nonholonomic system is described in \cite{Zenkov}. 
 The balanced sleigh has the mass over the point of contact and is equivalent to the skate problem, see Section~\ref{sec:skate}. 
It was proved in \cite{Zenkov} that trajectories of the sleigh's contact point  in the presence of a  heavy string without resistance are identical to those in
the absence of the  string. The latter are known to be uniform circular  or straight line motions (where the line could be regarded as a circle of infinite radius), cf. Section~\ref{sect:skate} and see e.g. \cite{bloch2003nonholonomic}. 
Each point of the string follows the trajectory of the contact point of
the sleigh with a suitable delay. This implies that after some time interval, the inertial dynamics of this sleigh-string system is represented by periodic trajectories in the phase space, and hence demonstrates integrable behavior \cite{Zenkov}.

\begin{remark}
The “snake” constraint -- namely, enforcing that $\partial_t z$ remains parallel to $\partial_s z$ -— can be realized by either the elastic energy or frictional dissipation (or both) penalizing motion transverse to the string. In analogy with Section~\ref{sec:skate}, we expect the emergence of an interpolating function $\mu(s)$ representing the ratio of dissipation to elastic resistance at each point $s$ along the string. The Lagrange--d'Alembert and vakonomic regimes correspond to the limiting behaviors $\mu(s) \to \infty$ and $\mu(s) \to 0$, respectively.
\end{remark}



\bigskip

{\bf Acknowledgments.} 
We are indebted to Anthony Bloch, Richard Montgomery, and Sina Nabizadeh
for fruitful discussions. The research of BK was partially supported by an NSERC Discovery Grant. 
The research of AGA was supported by the National Science Foundation under Grant NSF DMR–2116767.

\newpage
\small

\bibliographystyle{alpha}
\bibliography{vakonomics}

\end{document}